\documentclass[11pt]{nyjm}
     
\usepackage{amsmath}
\usepackage{amsthm}
\usepackage{amsopn}
\usepackage{amssymb}
\usepackage[cmtip, all]{xy}
\usepackage{stmaryrd}

\theoremstyle{plain}
\newtheorem{Lem}{Lemma}[section]
\newtheorem{Prop}[Lem]{Proposition}

\newtheorem{Thm}[Lem]{Theorem}

{\theoremstyle{definition} 

\newtheorem{Rk}[Lem]{Remark}
\newtheorem{Def}[Lem]{Definition}} 

\title{Homotopy fixed points for 
profinite groups emulate homotopy 
fixed points for discrete groups}

\author{Daniel G. Davis}
\address{Department of Mathematics, University of Louisiana at Lafayette, Lafayette, LA 70504, U.S.A.} 

\email{dgdavis@louisiana.edu, www.ucs.louisiana.edu/$\sim$dxd0799/ }    



\DeclareMathOperator*{\holim}{holim}
\DeclareMathOperator*{\holimG}{holim^{\mathit{G}}}
\DeclareMathOperator*{\holimH}{holim^{\mathit{H}}}
\DeclareMathOperator*{\hocolim}{hocolim}
\DeclareMathOperator*{\colim}{colim}

\newcommand{\zig}{\addtocounter{Lem}{1}\tag{\theLem}} 
 
\renewcommand{\:}{\colon}

\newcommand{\emhyph}{--\penalty0\hskip0pt\relax}

\begin{document}

\begin{abstract}
If $K$ is a discrete group and $Z$ is a $K$--spectrum, then 
the homotopy fixed point spectrum 
$Z^{hK}$ is $\mathrm{Map}_\ast(EK_+, Z)^K,$ the 
fixed points of a familiar expression. Similarly, if $G$ is a 
profinite group and $X$ is a discrete $G$--spectrum,
then $X^{hG}$ is often given by 
$(\mathcal{H}_{\scriptscriptstyle{G,X}})^{G}$, where 
$\mathcal{H}_{\scriptscriptstyle{G,X}}$ is a 
certain explicit construction given by a 
homotopy limit in the category of discrete $G$--spectra.
Thus, in each of two common equivariant settings, 
the homotopy fixed point 
spectrum is equal to the fixed points of an explicit 
object in the ambient equivariant category.
We enrich this pattern by proving in a precise sense 
that the discrete $G$--spectrum 
$\mathcal{H}_{\scriptscriptstyle{G,X}}$ is 
just ``a profinite version" of $\mathrm{Map}_\ast(EK_+, Z)$: at each 
stage of its construction, $\mathcal{H}_{\scriptscriptstyle{G,X}}$ 
replicates in the setting of discrete $G$-spectra 
the corresponding stage in the formation of $\mathrm{Map}_\ast(EK_+, Z)$ 
(up to a certain natural identification).
\end{abstract}
\maketitle
\tableofcontents
\section{Introduction}
\subsection{Recalling a familiar scenario: homotopy fixed points for discrete groups}
\label{subsection}
\par
Let $K$ be a discrete group and let $Z$ be a (naive) $K${\emhyph}spectrum, where, 
here and everywhere else in this paper (unless explicitly stated otherwise), 
``spectrum" means 
Bousfield-Friedlander spectrum of simplicial sets. Let $EK$ be the 
usual simplicial set with $n$--simplices equal to the cartesian product 
$K^{n+1}$, for each $n \geq 0$; let $EK_+$ denote $EK$ with a 
disjoint basepoint added; and let \[(-)_{f} \: \mathrm{Spt} 
\to \mathrm{Spt}, \ \ \ Y \mapsto Y_f\] be a fibrant replacement functor for the 
model category of spectra (with the usual stable structure). Also, 
given a pointed simplicial set $L$ and any spectrum $Y$, let 
$\mathrm{Map}_\ast(L, Y)$ be the mapping spectrum with $m$th 
pointed simplicial set $\mathrm{Map}_\ast(L,Y)_m$ having 
$n$--simplices equal to 
\[\mathrm{Map}_{\mathcal{S}_\ast}(L, Y_m)_n = \mathcal{S}_\ast(L 
\wedge \Delta[n]_+, Y_m),\] where $\mathcal{S}_\ast$ is the category 
of pointed simplicial sets. Then 
the homotopy fixed point spectrum $Z^{hK}$ 
is given explicitly by 
\[Z^{hK} = \mathrm{Map}_\ast(EK_+, Z_{f})^K.\] One reason for the importance 
of the explicit construction $\mathrm{Map}_\ast(EK_+, Z_{f})^K$ is that it 
makes it possible to build the descent spectral sequence 
\[E_2^{s,t} = H^s(K; \pi_t(Z)) \Longrightarrow \pi_{t-s}(Z^{hK}).\]
\subsection{Considering homotopy fixed points for profinite groups: a pattern 
emerges}\label{subsectiontwo}
Now let $G$ be a profinite group, let $\mathrm{Spt}_G$ be the 
simplicial model category of discrete $G$--spectra (for details, we refer 
the reader to \cite[Section 3]{cts} 
and \cite[Remark 3.11]{GJlocal}), and 
let $X \in \mathrm{Spt}_G$. We consider how to 
carry out the above constructions for $K$ and $Z$ in this 
profinite setting. 
\begin{Rk}
In the titles of this paper and \S \ref{subsectiontwo}, the phrase 
``homotopy fixed points for profinite groups" is meant for the 
setting of discrete $G$-spectra. We point out that 
there is a theory of homotopy 
fixed points for profinite $G$-spectra (see \cite{hfplt}) and our phrasing is 
not meant to be exclusionary.
\end{Rk}
\par
As explained in 
\cite[Definition 7.1]{cts}, the functor 
\[\mathrm{Map}_c(G,-) \: \mathrm{Spt}_G \to \mathrm{Spt}_G, \ \ \ 
X \mapsto \mathrm{Map}_c(G,X),\] where each pointed simplicial discrete 
$G$--set $\mathrm{Map}_c(G, X)_m$ satisfies 
\[(\mathrm{Map}_c(G, X)_m)_n = \mathrm{Map}_c(G, (X_m)_n)\] 
(the set of continuous functions $G \to (X_m)_n$), 
forms a triple, and hence, there is a cosimplicial 
discrete $G$--spectrum $\mathrm{Map}_c(G^\bullet,X)$, whose 
$l${\emhyph}cosimplices are obtained by applying $\mathrm{Map}_c(G,-)$ 
iteratively to $X$,  $l+1$ times. Thus, there is an isomorphism
\[\mathrm{Map}_c(G^\bullet, X)^l \cong \mathrm{Map}_c(G^{l+1}, X)\] 
of discrete $G$--spectra. Also, by \cite[Lemma 2.1]{agt}, the map 
\[X \xrightarrow{\,\cong\,} \colim_{N \vartriangleleft_o G} X^N \to 
\colim_{N \vartriangleleft_o G} \,(X^N)_f =: \widehat{X}\] is a weak equivalence 
in $\mathrm{Spt}_G$, with target $\widehat{X}$ fibrant in $\mathrm{Spt}$. 
\par
We let $X^{hG}$ denote the output of the total right derived functor of fixed points 
$(-)^G \: \mathrm{Spt}_G \to \mathrm{Spt}$, when applied to $X$: the spectrum 
$X^{hG}$ is more succinctly known as the homotopy fixed point spectrum of $X$ with 
respect to the continuous action of $G$. Also, 
let $\holimG$ denote the homotopy limit for $\mathrm{Spt}_G$, 
as defined in \cite[Definition 18.1.8]{hirschhorn}, and let 
$H^s_c(G; M)$ be equal to the continuous cohomology of $G$ with 
coefficients in the discrete $G$--module $M$. Then 
by \cite[Theorem 7.2]{agt} and 
\cite[Theorem 2.3, proof of Theorem 5.2]{fibrantmodel}, there is a 
weak equivalence 
\[X^{hG} \overset{\simeq}{\longrightarrow} \bigl(\holimG_\Delta 
\mathrm{Map}_c(G^\bullet, \widehat{X})\bigr)^{\negthinspace G},\] 
whenever 
any one of the following conditions holds:
\begin{enumerate}
\item[(i)]
$G$ has finite virtual cohomological dimension 
(that is, $G$ contains an open subgroup $U$ such 
that $H^s_c(U; M) = 0$, for all $s>u$ and all discrete $U$--modules 
$M$, for some integer $u$);
\item[(ii)]
there exists a fixed integer $p$ such that $H^s_c(N; \pi_t(X)) = 0$, for all 
$s>p$, all $t \in \mathbb{Z}$, and all 
$N \vartriangleleft_o G$; or
\item[(iii)]
there exists a fixed integer $r$ such that $\pi_t(X) = 0$, for all $t > r.$
\end{enumerate} 

As in the case of $Z^{hK}$, one of the main reasons why the explicit 
construction \[\bigl({\holim_\Delta}^G 
\,\mathrm{Map}_c(G^\bullet, \widehat{X})\bigr)^{\negthinspace G} 
= \Bigl(\mspace{1mu}\colim_{N \vartriangleleft_o G} \,\bigl(\holim_\Delta 
\mathrm{Map}_c(G^\bullet, 
\widehat{X})\bigr)^{\negthinspace N}\Bigr)^{\negthinspace 
\negthinspace \mspace{1mu} G}\] 
(see \cite[Theorem 2.3]{fibrantmodel}; the ``holim" denotes the homotopy 
limit for spectra) 
is important is that when $X$ satisfies one of the above conditions, 
the construction makes it possible to build the descent 
spectral sequence
\begin{equation}\zig\label{dss}
E_2^{s,t} = H^s_c(G; \pi_t(X)) \Longrightarrow \pi_{t-s}(X^{hG})
\end{equation} (as in \cite[Theorem 7.9]{cts}, by using (\ref{usedabove}) 
below: given the context, 
this reference is the most immediate source for the derivation of (\ref{dss}), but  
the account in \cite[Theorem 7.9]{cts} is just a particular case of the much earlier 
\cite[Proposition 1.36]{thomason}, and, in the literature for 
``simplicial-set-based discrete $G$--objects," the references 
\cite[Corollary 3.6]{jardinejpaa}, 
\cite[Section 5]{hGal}, \cite[(6.7)]{Jardine}, and \cite[Section 2.14]{stavros} 
are earlier than \cite[Theorem 7.9]{cts} and contain all of its key 
ingredients).

Given the above discussion, it is natural to make the following definition.

\begin{Def}
If the discrete $G$--spectrum $X$ satisfies any one of the conditions 
(i), (ii), and (iii) above, then we say that $X$ is a {\em concrete} discrete 
$G$--spectrum, since $X$ has a concrete model for its homotopy fixed 
point spectrum.
\end{Def}

\par
In practice, at least one of the above three conditions is usually satisfied. For example, as 
is common in 
chromatic homotopy theory, let $\Gamma_h$ be equal to 
any formal group law of 
height $h$, with $h$ positive, over $k$, a finite field of prime characteristic $p$ 
that contains the field $\mathbb{F}_{p^h}$, and consider any closed subgroup $H$ of the 
compact $p$-adic analytic group 
$G(k, \Gamma_h)$, the extended Morava stabilizer group (see 
\cite[Section 7]{Pgg/Hop0}). Then $H$ is a profinite group with 
finite virtual cohomological dimension (see 
\cite[Section 2.2.0]{Morava}), and thus, the discrete $H$--spectrum 
\begin{equation}\zig\label{neatdiscreteguy}
\Bigl(\mspace{1mu}\colim_{N \vartriangleleft_o G(k, \Gamma_h)} E(k, \Gamma_h)^{hN}\Bigr) 
\wedge F \simeq 
E(k, \Gamma_h) \wedge F\end{equation} satisfies condition $(i)$ above. In 
(\ref{neatdiscreteguy}), $E(k, \Gamma_h)$ is the Morava $E$--theory 
associated to the pair $(k, \Gamma_h)$ (again, see \cite[Section 7]{Pgg/Hop0}); 
the construction of the homotopy fixed point spectrum $E(k, \Gamma_h)^{hN}$ 
is described in \cite[page 2895]{iterated} (with substantial input from 
\cite{DH} and \cite[Theorem 8.2.1]{joint}); $F$ is any finite spectrum 
of type $h$; and the weak equivalence 
is due to \cite{DH} (the details are written out in \cite[Theorem 6.3, Corollary 6.5]{cts}). Discrete $H$--spectra that have the form 
given by (\ref{neatdiscreteguy}) are the building blocks for many of the 
continuous $H$--spectra that are studied in chromatic theory; for examples, 
see \cite[Section 9]{brsbrauer}, \cite[Section 2.3]{modular}, and 
\cite[pages 153--155]{lawsontopology}. 
 
If $Y$ is any spectrum and $X$ is the discrete $G$--spectrum 
$\mathrm{Map}_c(G, Y)$, then by \cite[Lemma 9.4.5]{Wilson}, condition (ii) above 
is satisfied, with $p = 0$. Such concrete 
discrete $G$--spectra arise in the theory \cite{Rognes} 
of Galois extensions for commutative rings in stable homotopy theory: 
for example, if $T$ is a spectrum such that the Bousfield 
localization $L_T(-)$ is smashing, $M$ is any finite spectrum, 
$k$ is a spectrum such that $L_k(-) \simeq L_ML_T(-)$, and (for the 
remainder of this sentence, using symmetric spectra as needed) $E$ is a 
$k$--local profinite $G$--Galois extension of a $k$--local commutative symmetric 
ring spectrum $A$, then 
\[L_k(E \wedge_A E) \simeq L_k\bigl(\mathrm{Map}_c(G, E)\bigr),\] 
by \cite[Proposition 6.2.4]{joint}.

We see that under hypotheses that are often satisfied, 
the homotopy fixed point spectrum $X^{hG}$ can be obtained by 
taking the $G$--fixed points of the discrete $G$--spectrum 
\[{\holim_\Delta}^G 
\,\mathrm{Map}_c(G^\bullet, \widehat{X}) = 
\colim_{N \vartriangleleft_o G} \,\bigl(\holim_\Delta 
\mathrm{Map}_c(G^\bullet, 
\widehat{X})\bigr)^{\negthinspace N}\] (this is the 
discrete $G$--spectrum $\mathcal{H}_{\scriptscriptstyle{G,X}}$ that is 
referred to in the 
abstract for this paper), and hence, the construction 
of $X^{hG}$ follows a pattern that was seen before in the case of 
$Z^{hK}$: form the homotopy fixed point spectrum by taking 
the fixed points of an explicitly constructed 
spectrum that is an object in the equivariant category of spectra 
that is under consideration ($\mathrm{Spt}_G$ or $K$--spectra, 
respectively).

But there is more to 
the above pattern than just the last observation: 
this is hinted at by the tandem facts that, as in \cite[proof of Theorem 5.2]{fibrantmodel}, 
there is an isomorphism
\begin{equation}\label{usedabove}\zig
\bigl(\displaystyle{\holim_\Delta}^G \,\mathrm{Map}_c(G^\bullet, \widehat{X})\bigr)^{\negthinspace G} 
\cong \bigl(\displaystyle{\holim_\Delta \mathrm{Map}_c(G^\bullet, \widehat{X})}\bigr)^{\negthinspace G}\end{equation}
and the $G$--spectrum 
$\holim_{\Delta} \mathrm{Map}_c(G^\bullet, \widehat{X})$  
on the right-hand side is often 
viewed as being ``a profinite version" of 
the construction $\mathrm{Map}_\ast(EK_+, Z_f)$ (for example, see \cite{mitchell}). (Also, it is worth pointing out that if $X$ is a 
concrete discrete $G$--spectrum, then $X^{hG}$ has almost always been presented in the literature as being the $G$--fixed points of 
$\holim_\Delta \mathrm{Map}_c(G^\bullet, \widehat{X})$ (this 
$G$--spectrum  
is not, in general, 
a discrete $G$--spectrum: see the remark below for an example of 
when this happens), instead of as the 
$G$--fixed points of 
$\textstyle{\holim^G_\Delta \mathrm{Map}_c(G^\bullet, \widehat{X})}$.) 
However, what the 
last assertion means has never been explained in a precise and systematic way, and 
further, as the above considerations make clear, it is rather 
$\textstyle{\holim^G_\Delta \mathrm{Map}_c(G^\bullet, \widehat{X})}$, 
instead of $\holim_\Delta \mathrm{Map}_c(G^\bullet, \widehat{X})$, 
that we want 
to understand as a ``profinite version" 
of $\mathrm{Map}_\ast(EK_+, Z_f)$. Thus, in 
this paper, we give a careful explanation of how 
$\textstyle{\holim^G_\Delta \mathrm{Map}_c(G^\bullet, \widehat{X})}$ 
is indeed a profinite version of 
$\mathrm{Map}_\ast(EK_+, Z_f)$. Rather than cluttering our introduction with an excess of definitions, we 
 refer the reader to Section \ref{last} for the exact details of this 
 explanation. 
 \begin{Rk} We pause to give an example of 
 $\holim_\Delta \mathrm{Map}_c(G^\bullet, \widehat{X})$ failing to 
 be a discrete $G$--spectrum. All of the following details are 
 expanded upon in \cite[Appendix A]{iterated}. Given distinct primes 
 $p$ and $q$, set \[G = \mathbb{Z}/p \times \mathbb{Z}_q\] (a profinite 
 group of finite virtual 
 cohomological dimension) and 
 let \[X = \textstyle{\bigvee_{n \geq 0}} \,\Sigma^nH(\mathbb{Z}/p[
 \mathbb{Z}/{q^n}]),\] a discrete $G$--spectrum. Then suppose that $\holim_\Delta \mathrm{Map}_c(G^\bullet, \widehat{X})$ is 
 a discrete $G$--spectrum: its $(\mathbb{Z}/p)$--fixed point spectrum 
 $\holim_\Delta \bigl(\mathrm{Map}_c(G^\bullet, \widehat{X})\bigr)^{\mathbb{Z}/p}$ is a discrete $\mathbb{Z}_q$--spectrum, 
 and hence, $\pi_0\bigl(\holim_\Delta \bigl(\mathrm{Map}_c(G^\bullet, \widehat{X})\bigr)^{\mathbb{Z}/p}\bigr)$ is a discrete $\mathbb{Z}_q$--module, 
 a contradiction.
 \end{Rk}
 \par
 For now, we 
 summarize our explanation with the following: it turns out that 
the ``co-steps" in the construction of 
 $\textstyle{\holim^G_\Delta \mathrm{Map}_c(G^\bullet, \widehat{X})}$ 
 are essentially identical to those involved in the 
 construction of a certain $K$--spectrum 
 $\widetilde{Z_K}$ that is 
 equivalent to $\mathrm{Map}_\ast(EK_+, Z_f)$, 
 except that 
 when imitating the construction of $\widetilde{Z_K}$, 
 at each co-step, if one obtains a $G$--spectrum that need not be, in general, 
 a discrete $G$--spectrum, then one makes it a discrete $G$--spectrum 
 in ``the canonical way," by applying the discretization functor (see 
 Definition \ref{discretization}).
 \par
In Section \ref{sec:two}, we define the $K$--spectrum $\widetilde{Z_K}$ 
and show that it is equivalent 
to $\mathrm{Map}_\ast(EK_+, Z_f)$: this reduces our task to relating 
$\textstyle{\holim^G_\Delta \mathrm{Map}_c(G^\bullet, \widehat{X})}$ 
to $\widetilde{Z_K}$. We do not claim any originality for Section \ref{sec:two} 
and we note that $\widetilde{Z_K}$ is closely related 
to the homotopy limit that is used in \cite[second half of page 226]{mitchell} 
to describe $Z^{hK}$. (The main difference between our presentation 
and that of \cite{mitchell} is that the object in \cite{mitchell} 
that plays the role of 
our $K^\bullet$ (below, in Section \ref{sec:two}) is defined differently.)  
\subsection{The pattern and the cases of compact Lie groups and 
profinite $G$-spectra} 
 
Let $H$ be a discrete or profinite group and let $Z$ be an object 
in the corresponding category $\mathtt{Sp}_H$ 
of $H$--spectra: if $H$ is discrete, then $\mathtt{Sp}_H$ is the category of 
naive $H$--spectra considered at the beginning of this Introduction, and 
if $H$ is profinite, then $\mathtt{Sp}_H$ is the full subcategory 
of $\mathrm{Spt}_H$ that consists of the concrete discrete $H$--spectra. In 
both cases, as 
recalled at the beginning and by our 
main result, respectively, there is the following pattern: 
the homotopy fixed 
point spectrum $Z^{hH}$ 
can always be formed by taking the $H$--fixed points of 
some construction ``$\mathrm{Map}_\ast(EH_+, Z)$" (the particular version of the 
spectrum ``$\mathrm{Map}_\ast(EH_+, Z)$" that is used depends 
on the case) that is an object in the category 
$\mathtt{Sp}_H$. 
\begin{Rk}
It was just noted that when $H$ is profinite, the appropriate version 
of ``$\mathrm{Map}_\ast(EH_+, Z)$" is not just a discrete $H$--spectrum, 
but it is also concrete (that is, a concrete discrete $H$--spectrum). 
This can be justified as follows: because $Z$ is concrete, the
$H$--equivariant map \[Z \overset{\simeq}{\longrightarrow} 
\colim_{N \vartriangleleft_o H} \bigl(\holim_\Delta 
\mathrm{Map}_c(H^\bullet, \widehat{Z})\bigr)^{\negthinspace N} 
= \holimH_\Delta \mathrm{Map}_c(H^\bullet, \widehat{Z})\] is 
a weak equivalence of spectra (by \cite[proof of Theorem 4.2]{fibrantmodel} and 
\cite[Theorem 7.2]{agt}; see \cite[page 145]{fibrantmodel} for 
the definition of the map), and hence, the target of the weak equivalence 
(which is the appropriate version 
of ``$\mathrm{Map}_\ast(EH_+, Z)$") 
is concrete (since the homotopy groups of the source and target of the 
weak equivalence are isomorphic as discrete $H$-modules), as desired.  
\end{Rk}
The above pattern also occurs when $H$ is a compact Lie group  
and $\mathtt{Sp}_H$ is the category of naive $H$--equivariant 
spectra (in the context of \cite{lms}): in this case, $Z^{hH}$ is the $H$--fixed points of the 
naive $H$--equivariant spectrum $F(EH_+, Z)$.

Now we again let $H$ be a profinite group and set 
$\mathtt{Sp}_H$ equal to the 
category 
of profinite $H$--spectra, as defined in \cite{Quick}. 
Interestingly, 
we will see that in this case, the above pattern 
does not 
go through all the way. By 
\cite[Remark 3.8, Definition 3.14]{hfplt}, 
$Z^{hH}$ is the $H$--fixed points of the explicit $H$--spectrum 
$\mathrm{Map}(EH, R_HZ)$, where here, $EH$ is regarded as a simplicial profinite 
$H$--set and $R_HZ$ is a functorial fibrant replacement 
of $Z$ in the stable model category $\mathtt{Sp}_H$. Also, the $H$--spectrum 
$\mathrm{Map}(EH,R_HZ)$ is defined as follows: 
for each $m \geq 0$,
\[\mathrm{Map}(EH,R_HZ)_m = \mathrm{map}_{\hat{\mathcal{S}}_\ast}(EH_+, 
(R_HZ)_m),\] where the right-hand side is an instance of the simplicial mapping space 
for the category $\hat{\mathcal{S}}_\ast$ of pointed simplicial profinite sets. 
Thus, in agreement with the pattern, 
the construction $\mathrm{Map}(EH,R_HZ)$ is
indeed a version of ``$\,\mathrm{Map}_\ast(EH_+,Z)$."  

In contrast with the pattern, however, it turns 
out that $\mathrm{Map}(EH,R_HZ)$ is not, in general, a profinite $H$--spectrum. 
To see that this is true, suppose that 
$\mathrm{Map}(EH,R_HZ)$ is always a profinite $H$--spectrum. Then 
\[Z^{hH} =\mathrm{Map}(EH,R_HZ)^H = \lim_H \mathrm{Map}(EH,R_HZ),\] 
where the last expression is a limit in the category of spectra. Since 
the forgetful functor from profinite spectra to spectra is a right adjoint (see 
\cite[Proposition 4.7]{Quick}), limits in profinite spectra are formed in spectra, 
and thus, since $\mathrm{Map}(EH,R_HZ)$ is a profinite $H$--spectrum, the 
above limit can be regarded as a limit in the category of profinite spectra. 
It follows that $Z^{hH}$ must be a profinite spectrum. But $Z^{hH}$ 
is not always a profinite spectrum, by \cite[page 194; Remark 3.16]{hfplt} (see 
also the helpful discussion between Proposition 2.15 and Theorem 2.16 in 
\cite{QuickSpaces}), showing that 
$\mathrm{Map}(EH, R_HZ)$ is not always a profinite 
$H$--spectrum. 
 
We continue to let $H$ be profinite. 
With various properties of the theories of homotopy fixed points for 
discrete and profinite $H$--spectra laid out on the table, it is worth making 
the following observation: in these theories, abstract and explicit realizations 
of homotopy fixed points do not go easily together. In the world of discrete 
$H$--spectra, the homotopy fixed point spectrum is abstractly 
defined as the right derived functor of fixed points, but only when certain 
hypotheses are satisfied, is the homotopy fixed points known to be given by a 
concrete model. In the setting of profinite $H$--spectra, the 
situation is reversed: the homotopy fixed 
points are always given by an explicit model (that is, $\mathrm{Map}(EH, 
R_HZ)^H$, as considered above), but in general, the homotopy fixed points are not 
the right derived functor of fixed points. To see this last point, suppose 
that $Z$ is a profinite $H$--spectrum with $Z^{hH} = (R_HZ)^{H}$. 
Then, by repeating 
an argument that was used above, $Z^{hH} = 
\lim_H R_HZ$ must be a profinite spectrum. Since $Z^{hH}$ 
is not always a profinite spectrum (see above), $Z^{hH}$ 
cannot in general be defined abstractly as the 
output of the right derived functor of fixed points.  
      
\par
We conclude the Introduction with a few comments about our 
notation. We use $\mathcal{S}$ to denote the category of simplicial sets. 
Given a set $S$, we 
let $c_\bullet(S)$ denote the 
constant simplicial set on $S$, and by a slight abuse of this notation, we use 
$c_\bullet(\ast)$ to denote the constant 
simplicial set on the set $\{\ast\}$ that consists of a single point. To avoid 
any possible confusion, we note that $c_\bullet(S)_+$ is $c_\bullet(S)$ 
with a disjoint basepoint added.  
 
\subsection*{Acknowledgements}
I found the main result in this paper during 
the course of discussions with Markus Szymik. I thank Markus 
for these stimulating exchanges. Also, I thank the referee for helpful comments.
\section{$K$--spectrum $\widetilde{Z_K}$ is equivalent to $\mathrm{Map}_\ast(EK_+, Z_f)$}
\label{sec:two}
\par
Recall from \S \ref{subsection} the $K$--spectrum 
$\mathrm{Map}_\ast(EK_+, Z_f)$: for each $n \geq 0$, $K$ acts diagonally 
on the $n$--simplices $K^{n+1}$ of $EK$ and the mapping spectrum 
has its $K$--action induced by conjugation on the level of sets (that is, 
by the formula \[(k \cdot f_j)(k_1, k_2, ..., k_{j+1}) = 
k \cdot f_j(k^{-1}\cdot(k_1, k_2, ..., k_{j+1})),\] where $k, k_1, ..., k_{j+1} 
\in K$ and 
\[\{f_j \: K^{j+1} \to \mathrm{hom}_\ast(\Delta[n], (Z_f)_m)_j\}_{j \geq 0} 
\in \mathcal{S}(EK, \mathrm{hom}_\ast(\Delta[n], (Z_f)_m)),\] 
with 
\[\mathcal{S}(EK, \mathrm{hom}_\ast(\Delta[n], (Z_f)_m)) \cong \mathrm{Map}_{\mathcal{S}_\ast}(EK_+, (Z_f)_m)_n,\] $\mathrm{hom}_\ast(\Delta[n], (Z_f)_m)$ 
is a cotensor in $\mathcal{S}_\ast$, and 
$K$ acts only on $(Z_f)_m$ in the expression 
$\mathrm{hom}_\ast(\Delta[n], (Z_f)_m)$). 

\begin{Def}
Let $K^\bullet$ be the canonical bisimplicial set 
\[K^\bullet \: \Delta^\mathrm{op} \to \mathcal{S}, \ \ \ [n] \mapsto 
(K^\bullet)_n = c_\bullet(K^{n+1}),\] 
with $\mathrm{diag}(K^\bullet) = EK,$ where 
$\mathrm{diag}(K^\bullet)$ is the diagonal of $K^\bullet$. 
\end{Def}

\par
Given a 
simplicial set $L$ and a spectrum $Y$, we write $Y^L$ for the cotensor 
in the simplicial model category $\mathrm{Spt}$. It will be helpful to 
note that \[Y^L = \mathrm{Map}_\ast(L_+, Y).\]

\begin{Def}
Notice that $\displaystyle{\hocolim_{\Delta^\mathrm{op}} K^\bullet} 
\equiv \hocolim_{[n] \in \Delta^\mathrm{op}}\,(K^\bullet)_n$. 
There is an isomorphism $(Z_f)^{(\hocolim_{\Delta^\mathrm{op}} K^\bullet)} \xrightarrow{\,\cong\,} 
\holim_\Delta \, (Z_f)^{(K^\bullet)}$ and the target of this map 
is defined to be the $K$--spectrum $\widetilde{Z_K}$. Thus, we have
\[\widetilde{Z_K} = \holim_\Delta \, (Z_f)^{(K^\bullet)}.\]
\end{Def}

As alluded to in the 
Introduction, the following result -- or at least some version of it -- seems 
to be well-known, but for the sake of completeness, we give a 
proof of the precise version that we need.

\begin{Thm}\label{keyw.e.}
There is a canonical $K$--equivariant map
\[\mathrm{Map}_\ast(EK_+, Z_f) \xrightarrow{\,\simeq\,} \widetilde{Z_K}\] 
that is a weak equivalence in $\mathrm{Spt}$.
\end{Thm}

\begin{proof}
Since $\mathrm{Map}_\ast(EK_+, Z_f)$ is the cotensor $(Z_f)^{EK}$, 
it suffices to construct a canonical $K$--equivariant map 
$(Z_f)^{EK} \to (Z_f)^{(\hocolim_{\Delta^\mathrm{op}} K^\bullet)}$ 
that is a weak equivalence of spectra. 
Notice that there is the composition
\[
\widetilde{\phi_\ast} \: 
\displaystyle{\hocolim_{\Delta^\mathrm{op}} K^\bullet} 
\xrightarrow{\,\simeq\,} 
|K^\bullet| \xrightarrow{\,\,\cong\,\,} \mathrm{diag}(K^\bullet) = EK
\] 
of canonical $K$--equivariant 
maps, with the first map, \[\phi_\ast \: \hocolim_{\Delta^\mathrm{op}} 
K^\bullet \xrightarrow{\,\simeq\,} |K^\bullet|\] (our label for this map comes from 
\cite[Corollary 18.7.5]{hirschhorn}, where this map is referred to as 
``the Bousfield-Kan map"), and the second map equal to 
a weak equivalence and 
an isomorphism (as labeled above), respectively. Then the desired 
map is just $(Z_f)^{\mspace{1mu}
\widetilde{\phi_\ast}
}$ and we only need to show 
that this map is a weak equivalence: to do this, since a strict weak 
equivalence of spectra is a (stable) weak equivalence, it suffices to show that 
for each $m \geq 0$, the map 
\begin{align*}\mathrm{Map}_\ast&(|K^\bullet|_+, Z_f)_m 
\\ & = \mathrm{Map}_{\mathcal{S}_\ast}(|K^\bullet|_+, (Z_f)_m) 
\to \mathrm{Map}_{\mathcal{S}_\ast}\bigl(\bigl(\hocolim_{\Delta^\mathrm{op}} 
K^\bullet\bigr)_{\mspace{-2mu}+}, (Z_f)_m\bigr)\end{align*} is a weak equivalence in $\mathcal{S}$.

If $L$ and $L'$ are simplicial sets, then $L_+ \wedge (L')_+ \cong 
(L \times L')_+$, and hence, we only need to show that each map
\[\mathrm{Map}_{\mathcal{S}}(|K^\bullet|, (Z_f)_m) 
\to \mathrm{Map}_{\mathcal{S}}\bigl(\hocolim_{\Delta^\mathrm{op}} 
K^\bullet, (Z_f)_m\bigr)\] is a weak equivalence in $\mathcal{S}$: 
this follows from the fact that in $\mathcal{S}$, 
$\phi_\ast$ is a weak equivalence and 
$(Z_f)_m$ is fibrant.
\end{proof}
       
The equivalence in 
Theorem \ref{keyw.e.} implies that to relate the discrete $G$--spectrum 
$\textstyle{\holim^G_\Delta \mathrm{Map}_c(G^\bullet, \widehat{X})}$ 
to the $K$--spectrum 
$\mathrm{Map}_\ast(EK_+, Z_f)$, we can just as well compare 
$\textstyle{\holim^G_\Delta \mathrm{Map}_c(G^\bullet, \widehat{X})}$ 
to $\widetilde{Z_K}$. To do this comparison, it 
will be helpful to write $\widetilde{Z_K}$ 
a little differently: there are isomorphisms 
\begin{align*}
\widetilde{Z_K} & = 
\holim_{[n] \in \Delta} \,(Z_f)^{(c_\bullet(K^{n+1}))} \\ 
& \cong 
\holim_{[n] \in \Delta} 
\,(Z_f)^{\mspace{-1mu}(\mspace{1.7mu}\prod_{\ell \in \{1, 2, ..., n+1\}} 
\mspace{-2mu}c_\bullet(K))}
 \\ & \cong 
\holim_{[n] \in \Delta} \,\underbrace{(\cdots 
(((}_{n+1}Z_f\underbrace{)^{c_\bullet(K)})^{c_\bullet(K)}) 
\cdots)^{c_\bullet(K)}}_{n+1}.
\end{align*}
\section{Building $\mathrm{Map}_c(G, X)$ from fixed points of cotensors}

We begin this section by recalling that given $X \in \mathrm{Spt}_G$, the 
$G$--action on the discrete $G$--spectrum $\mathrm{Map}_c(G,X)$ 
is induced by the $G$--action on the level of sets that is defined by 
$(g \cdot (h_m)_n)(g') = (h_m)_n(g'g)$, where $g, g' \in G$ and, for 
each $m, n \geq 0$, $(h_m)_n \in \mathrm{Map}_c(G,(X_m)_n)$.

Notice that there are natural $G$--equivariant isomorphisms
\begin{align*}
\mathrm{Map}_c(G, X) & \cong \colim_{N \vartriangleleft_o G} 
\,\textstyle{\prod_{G/N} X} 
\\
& \cong \displaystyle{\colim_{N \vartriangleleft_o G}} 
\, \mathrm{Map}_\ast(\textstyle{\bigvee_{G/N}c_\bullet(\ast)_+, X)}\\
& \cong \displaystyle{\colim_{N \vartriangleleft_o G}} 
\, \mathrm{Map}_\ast(c_\bullet(G/N)_+, X),
\end{align*} 
where the last expression above uses the following convention.
\begin{Def}\label{def:action}
The spectrum $\mathrm{Map}_\ast(c_\bullet(G/N)_+, X)$ has a 
$G/N$--action that is determined 
by the formula $(g_1N \cdot f_j)(g_2N) = f_j(g_2g_1N),$ for $g_1, g_2 \in G$ and \[\{f_j\}_{j \geq 0} \in \mathcal{S}(c_\bullet(G/N), \mathrm{hom}_\ast(\Delta[n], X_m))\] (for example, see the beginning of Section \ref{sec:two}). 
\end{Def}

We have shown 
that there is a natural isomorphism 
\[\mathrm{Map}_c(G, X) \cong \colim_{N \vartriangleleft_o G} 
\,\mathrm{Map}_\ast(c_\bullet(G/N)_+, X)\] 
in $\mathrm{Spt}_G$; this observation was made in 
\cite[page 210]{hGal} in the context of simplicial discrete $G$--sets.
\begin{Prop}\label{cotensor}
If $N$ is an open normal subgroup of $G$, then 
there is a natural $G/N$--equivariant isomorphism 
\[\mathrm{Map}_\ast(c_\bullet(G)_+, X)^N \cong  
\mathrm{Map}_\ast(c_\bullet(G/N)_+, X)\] 
of $G/N$--spectra, 
where $\mathrm{Map}_\ast(c_\bullet(G)_+,X)$ has the $G$--action given by conjugation, $\mathrm{Map}_\ast(c_\bullet(G)_+, X)^N$ denotes 
the $N$--fixed point spectrum {\rm(}and not a cotensor{\rm )}, 
and $\mathrm{Map}_\ast(c_\bullet(G/N)_+, X)$ has the $G/N$--action 
given in Definition {\rm\ref{def:action}}.
\end{Prop} 
\begin{proof}
To verify this result, it suffices to show that on the level of simplices 
there is a natural $G/N$--equivariant isomorphism
\[\bigl(\mathrm{Map}_{\mathcal{S}_\ast}(c_\bullet(G)_+, 
X_m)^N\bigr)_{\mspace{-2mu}n} \cong  
\mathrm{Map}_{\mathcal{S}_\ast}(c_\bullet(G/N)_+, 
X_m)_n,\] and hence, we only 
need to show that there is a natural $G/N$--equivariant bijection
\[\mathcal{S}_\ast(c_\bullet(G)_+, \mathrm{hom}_\ast(\Delta[n], X_m))^N 
\cong \mathcal{S}_\ast(c_\bullet(G/N)_+, \mathrm{hom}_\ast(\Delta[n], X_m))\] 
of sets, where the $G$--action on 
$\mathcal{S}_\ast(c_\bullet(G)_+, \mathrm{hom}_\ast(\Delta[n], X_m))$ is 
such that $G$ only acts on $X_m$ in the cotensor 
$\mathrm{hom}_\ast(\Delta[n], X_m)$. 
\par
Since the functor $(-)_+ \: 
\mathcal{S} \to \mathcal{S}_\ast$ 
is left adjoint to the forgetful functor, our last assertion above 
is equivalent 
to there being a natural $G/N$--equivariant bijection
\[\mathcal{S}(c_\bullet(G), \mathrm{hom}_\ast(\Delta[n], X_m))^N 
\cong \mathcal{S}(c_\bullet(G/N), \mathrm{hom}_\ast(\Delta[n], X_m)).\] 
The existence of this $G/N$--equivariant bijection 
follows from the fact that if $W$ is any $G$--set, then, 
letting $\mathrm{Sets}$ denote the 
category of sets, the natural function 
\[\lambda \: \mathrm{Sets}(G, W)^N \to \mathrm{Sets}(G/N, W), 
\ \ \ f \mapsto \Bigl[\lambda(f) \: gN \mapsto g\cdot f(g^{-1})\Bigr]\] 
is a $G/N$--equivariant isomorphism. Here, of course, $G$ acts 
on $\mathrm{Sets}(G,W)$ by conjugation and the $G/N$--action 
on $\mathrm{Sets}(G/N, W)$ is defined by \[(g_1N \cdot h)(g_2N) 
= h(g_2g_1N), \ \ \  g_1, g_2 \in G, \ h \in \mathrm{Sets}(G/N, W).
\qedhere 
\]
\end{proof}
\par
By Proposition \ref{cotensor} and the 
discussion that precedes it, we immediately obtain the following result.
\begin{Prop}\label{various}
There is an isomorphism
\[\mathrm{Map}_c(G,X) \cong \colim_{N \vartriangleleft_o G} 
\mathrm{Map}_\ast(c_\bullet(G)_+, X)^N\] of discrete $G$--spectra.
\end{Prop}
\begin{Rk} 
For the duration of this remark, suppose that $G \neq 
\{{e}_{_{\scriptscriptstyle{G}}}\}$. 
The right-hand side of the isomorphism in 
Proposition \ref{various} can be written as the discrete $G$--spectrum 
$\colim_{N \vartriangleleft_o G} 
\bigl(X^{c_\bullet(G)}\bigr)^N,$ where $X^{c_\bullet(G)}$ is a 
cotensor for spectra. Interestingly, by 
\cite[proof of Theorem 2.3]{fibrantmodel}, the cotensor 
$\bigl(X^{c_\bullet(G)}\bigr)_{\negthinspace G}$ for the simplicial 
model category $\mathrm{Spt}_G$ can also be written as 
$\colim_{N \vartriangleleft_o G} 
\bigl(X^{c_\bullet(G)}\bigr)^N,$ where $G$ acts on $X^{c_\bullet(G)}$ 
by acting only on $X$. However, despite their cosmetic similarity, 
$\mathrm{Map}_c(G,X)$ and 
$\bigl(X^{c_\bullet(G)}\bigr)_{\negthinspace G}$ 
are, in general, 
not isomorphic as discrete $G$--spectra, because of their different $G$--actions. For example, suppose that 
$Y_1$ and $Y_2$ are discrete $G$--spectra, with each having the trivial 
$G$--action. Then 
\begin{align*}
\mathrm{Spt}_G\bigl(Y_1, \bigl((Y_2)^{c_\bullet(G)}\bigr)_G\bigr) 
& \cong \mathrm{Spt}_G\negthinspace\,\bigl(Y_1, \bigl((Y_2)^{(\coprod_{g \in G} c_\bullet(\ast))}\bigr)_G\bigr)
\\ & \cong \textstyle{\prod_{g \in G} \mathrm{Spt}_G(Y_1, Y_2)}\\ & \cong 
\textstyle{\prod_{g \in G} \mathrm{Spt}(Y_1, Y_2)}\end{align*} 
and 
\[\mathrm{Spt}_G(Y_1, \mathrm{Map}_c(G, Y_2)) \cong \mathrm{Spt}(Y_1, Y_2),\] and hence, 
$\bigl((Y_2)^{c_\bullet(G)}\bigr)_G$ and $\mathrm{Map}_c(G, Y_2)$ 
are, in general, not isomorphic as discrete $G$--spectra. 
\end{Rk}
\section{At each co-step, 
$\textstyle{\holim^G_\Delta \mathrm{Map}_c(G^\bullet, \widehat{X})}$ 
follows $\widetilde{Z_K}$ exactly, then makes the output into a discrete $G$--spectrum}  
\label{last}
\par
In Section \ref{sec:two}, we showed that there is a $K$-equivariant 
weak equivalence of spectra 
between $\mathrm{Map}_\ast(EK_+, Z_f)$ 
and 
\begin{equation}\zig\label{identify}
\widetilde{Z_K} = \holim_{[n] \in \Delta} \,\underbrace{(\cdots 
(((}_{n+1}Z_f\underbrace{)^{c_\bullet(K)})^{c_\bullet(K)}) 
\cdots)^{c_\bullet(K)}}_{n+1}.\end{equation}
We remark that (\ref{identify}) contains a slight 
abuse of notation: the equality in (\ref{identify}) is actually a 
natural identification between isomorphic $K$--spectra. Identity 
(\ref{identify}) is key to understanding the main result of this paper, but 
to explain this result, we need one more tool, given in 
Definition \ref{discretization} below. After some discussion 
of the functor recalled in this definition, we will explain 
the main result.

Throughout this section, $G$ denotes an arbitrary profinite group.

\subsection{The discretization functor for $G$--spectra}

As noted in 
\cite[Remark 2.2]{jointwithTylertwo}, the 
isomorphism \[W \cong \colim_{N \vartriangleleft_o G} W^N\] 
satisfied by every $W \in \mathrm{Spt}_G$ is the basic 
fact behind the following.

\begin{Def}[{\cite[Remark 2.2]{jointwithTylertwo}}]
\label{discretization}
Let $G\mathrm{-Spt}$ be the category of (naive) $G$--spectra. The 
right adjoint of the forgetful functor $U_G \: 
\mathrm{Spt}_G \to G\mathrm{-Spt}$ is the {\em discretization} functor 
\[(-)_d \: G\mathrm{-Spt} \to \mathrm{Spt}_G, \ \ \ Y \mapsto (Y)_d = 
\colim_{N \vartriangleleft_o G} Y^N;\] $(Y)_d$ is ``the discrete 
$G$--subspectrum" of the $G$--spectrum $Y$. The application of the 
functor $(-)_d$ is the canonical way 
to ``convert" $Y$ into a discrete $G$--spectrum (the author 
would like to mention 
that he learned part of this perspective on $(-)_d$ from 
\cite[the brief discussion of (1.2.2)]{hGal}). It goes without saying that if the $G$--spectrum $Y$ already is a discrete $G$--spectrum, then $(Y)_d \cong Y$.
\end{Def}

Since $\mathrm{Spt}$ is a combinatorial model category, the category 
$G\mathrm{-Spt}$, which is isomorphic to the diagram category of 
functors $\{\ast_G\} \to \mathrm{Spt}$ out of the one-object groupoid $\{\ast_G\}$ 
associated to $G$, has an injective model structure (for example, 
see \cite[Proposition A.2.8.2]{luriebook}) in which a morphism 
of $G$--spectra is a weak equivalence (cofibration) if and only if it is a 
weak equivalence (cofibration) in $\mathrm{Spt}$. Thus, 
the left adjoint $U_G \: \mathrm{Spt}_G \to G\mathrm{-Spt}$ preserves 
weak equivalences and cofibrations, giving the next result, which 
gives some homotopical content to the fact that the discretization 
functor $(-)_d$ is the most natural way to convert a $G$--spectrum into a 
discrete $G$--spectrum. 

\begin{Thm}
\label{quillenpair}
The functors $(U_G, (-)_d)$ are a Quillen pair. In particular, if $Y$ is a 
fibrant $G$--spectrum, $\colim_{N \vartriangleleft_o G} Y^N$ is a fibrant 
discrete $G$--spectrum.
\end{Thm} 

\begin{Rk}
It is well-known that, as with most combinatorial model categories 
that consist of objects built out of simplicial presheaves on the canonical 
site of finite discrete $G$--sets, it is not easy to produce fairly explicit
examples of fibrant discrete $G$--spectra (for example, see \cite[page 1049]{GJlocal} and \cite[Introduction]{fibrantmodel}), and thus, one example of the utility of 
Theorem \ref{quillenpair} is that it provides a tool for doing this. 
\end{Rk}

\begin{Rk} We make a well-known observation that is a preparatory comment 
for the next remark below. The left adjoint $\mathrm{Spt} \to G\mathrm{-Spt}$ 
that sends a spectrum to itself, but now regarded as a $G$--spectrum that 
is equipped with the trivial $G$--action, preserves weak equivalences and 
cofibrations, and hence, the right adjoint
\[\lim_{\{\ast_G\}}\,(-) \: G\mathrm{-Spt} \to \mathrm{Spt}, \ \ \ Y \mapsto 
\lim_{\{\ast_G\}} Y = Y^G\] is a 
right Quillen functor. It follows that if $Y \to Y_\mathtt{f}$ is a trivial cofibration 
to a fibrant object, in $G\mathrm{-Spt}$, then 
\[Y^{hG} = (Y_\mathtt{f})^G,\] the right derived functor of fixed points $(-)^G 
\: G\mathrm{-Spt} \to \mathrm{Spt}$ 
applied to $Y$, is the homotopy fixed point spectrum of $Y$.
\end{Rk}

\begin{Rk}
Theorem \ref{quillenpair} has the following curious consequence: if 
$Y$ is any $G$--spectrum and $Y \to Y_\mathtt{f}$ is a trivial cofibration to 
a fibrant object, in $G\mathrm{-Spt}$, then $(Y_\mathtt{f})_d$ is a fibrant 
discrete $G$--spectrum, and hence, there is a weak equivalence
\begin{equation}\zig
\label{w.e.}
Y^{hG} = (Y_\mathtt{f})^G \xrightarrow{\,\cong\,} 
((Y_\mathtt{f})_d)^G \xrightarrow{\,\simeq\,} 
(((Y_\mathtt{f})_d)_{fG})^G = ((Y_\mathtt{f})_d)^{hG},\end{equation}
where the isomorphism is as in 
\cite[proof of Theorem 2.3: top of page 141]{fibrantmodel} 
and the weak equivalence is obtained by taking the $G$--fixed points 
of the natural trivial cofibration 
$(Y_\mathtt{f})_d \xrightarrow{\,\simeq\,} ((Y_\mathtt{f})_d)_{fG}$ 
in $\mathrm{Spt}_G$ that is associated to a fibrant replacement functor 
$(-)_{fG} \: \mathrm{Spt}_G \to \mathrm{Spt}_G$. The weak equivalence 
in (\ref{w.e.}) shows that for any $G$--spectrum $Y$, the ``discrete homotopy 
fixed point spectrum" $Y^{hG}$ is equivalent to the ``profinite 
homotopy fixed point spectrum" $((Y_\mathtt{f})_d)^{hG}$. This conclusion is a 
``discrete analogue" of the fact that the homotopy fixed point spectrum for an 
arbitrary continuous $G$--spectrum $\holim_i X_i$ is equivalent to 
the ``profinite homotopy fixed points" 
$(\holim^G_i (X_i)_{fG})^{hG}$ 
of the discrete $G$--spectrum $\holim^G_i (X_i)_{fG}$ 
\cite[Corollary 2.6]{jointwithTylertwo}.  
\end{Rk}  

\subsection{The main result}

Now we are ready to give the main result of this paper. 
Let $X$ be any discrete $G$--spectrum. Notice that, by 
Proposition \ref{various}, there is an isomorphism 
\[\mathrm{Map}_c(G, X) \cong \bigl(X^{c_\bullet(G)}\bigr)_{\negthinspace d}\,,\] 
where $X^{c_\bullet(G)}$ is a $G$--spectrum with $G$--action given by 
conjugation. Also, we have 
\[
\displaystyle{\holimG_\Delta \mathrm{Map}_c(G^\bullet, \widehat{X})} = 
\Bigl(\holim_{[n] \in \Delta} 
\underbrace{\mathrm{Map}_c(G, ..., \mathrm{Map}_c(G, \mathrm{Map}_c(G,}
_{n+1} \widehat{X}\underbrace{)) 
\cdots )}_{n+1}\,\Bigr)_{\negthinspace d}\,.\] Repeated application of the first of the above two conclusions,  
to the second conclusion, yields an isomorphism
\begin{equation*} 
\holimG_\Delta \mathrm{Map}_c(G^\bullet, \widehat{X}) \cong
\Bigl(\holim_{[n] \in \Delta} \underbrace{\bigl(\bigl(\cdots \bigl(\bigl(\bigl(}_{2(n+1)-1}\bigl(\widehat{X}\bigr)\underbrace{
\underbrace{
\underbrace{
^{^{\negthinspace \scriptstyle{\negthinspace 
c_\bullet(G)}}}\bigr)_{\negthinspace d}
}_{\text{once}}
\mspace{1mu}\bigr)^
{c_\bullet(G)}\bigr)_{\negthinspace d}
}_{\text{twice}} 
\cdots \bigr){^{^{
\negthinspace 
\scriptstyle{c_\bullet(G)}}}}\bigr)_{\negthinspace d}}_{(n+1) \ \text{times}}\,\Bigr)_{\negthinspace \mspace{-1.5mu} d} 
\end{equation*} of discrete $G$--spectra.

We recall (\ref{identify}) for the purpose of comparing it with 
the above isomorphism: 
\begin{equation}\zig\label{identifyrepeated}
\widetilde{Z_K} = \holim_{[n] \in \Delta} \,\underbrace{(\cdots 
(((}_{n+1}Z_f\underbrace{\underbrace{\underbrace{)^{c_\bullet(K)}}_{\text{once}})^{c_\bullet(K)}}_{\text{twice}}) 
\cdots)^{c_\bullet(K)}}_{(n+1) \ \text{times}}\,.\end{equation} Now the desired 
conclusion is clear: the construction of the discrete $G$--spectrum 
$\holim^G_\Delta \mathrm{Map}_c(G^\bullet, \widehat{X})$ -- whose 
$G$--fixed points often (that is, whenever $X$ is a concrete discrete $G$--spectrum) serve as a model for the homotopy 
fixed point spectrum $X^{hG}$ -- follows exactly the construction of the 
$K$--spectrum $\mathrm{Map}_\ast(EK_+, Z_f)$ (modulo
a natural identification with the right-hand side of (\ref{identifyrepeated})), 
subject to the natural constraint that whenever following the 
construction of $\mathrm{Map}_\ast(EK_+, Z_f)$ yields a 
$G$--spectrum that is not necessarily in $\mathrm{Spt}_G$ (that is, after 
each formation of a cotensor that has the form $W^{c_\bullet(G)}$, 
for some discrete $G$--spectrum $W$, and after forming the 
homotopy limit in $\mathrm{Spt}$), one applies the discretization 
functor $(-)_d$.
    
\begin{Rk}
We consider the last observation above in slightly more detail. Recall 
that $G$ is any profinite group and let $W$ denote any object in 
$\mathrm{Spt}_G$ that is fibrant as a spectrum. 
Also, let $I$ be the directed set of finite subsets 
of $G$, partially ordered by inclusion. For any integer $t$, there are 
$G$--equivariant isomorphisms
\begin{align*}
\pi_t\bigl(W^{c_\bullet(G)}\bigr) & \cong \pi_t\bigl(\textstyle{\prod}_{G} W\bigr) 
\cong \prod_G \pi_t(W) \\ & \cong \lim_{(g_1, g_2, ..., g_k) \in I} 
\Bigl(\pi_t(W)_{g_1} \times \pi_t(W)_{g_2} \times \cdots \times \pi_t(W)_{g_k}\Bigr),
\end{align*} 
where 
each $\pi_t(W)_{g_i}$ denotes a copy of $\pi_t(W)$ indexed by $g_i$. Since 
the finite product $\pi_t(W)_{g_1} \times \pi_t(W)_{g_2} \times \cdots 
\times \pi_t(W)_{g_k}$ in the category of abelian groups coincides with the 
product in the category of discrete $G$--modules, we see that 
the $G$--module $\pi_t\bigl(W^{c_\bullet(G)}\bigr)$ is an inverse limit of discrete $G$--modules. 
Note that if $W^{c_\bullet(G)}$ is a discrete $G$--spectrum (or even just 
weakly equivalent in $G$--$\mathrm{Spt}$ to a discrete $G$--spectrum), then 
the ``pro-discrete" $G$--module 
$\pi_t\bigl(W^{c_\bullet(G)}\bigr)$ is a discrete $G$--module. For arbitrary 
$G$, the preceding conclusion is typically not true, and hence, $W^{c_\bullet(G)}$ 
is not, in general, a discrete $G$--spectrum, so that applying 
the functor $(-)_d$ to $W^{c_\bullet(G)}$ typically does not leave $W^{c_\bullet(G)}$ 
unchanged.
\end{Rk}


\begin{thebibliography}{99}

\bibitem{brsbrauer}
{\scshape Baker, Andrew; Richter, Birgit; Szymik, Markus.}
\newblock Brauer groups for commutative {$S$}-algebras.
\newblock {\it J. Pure Appl. Algebra} {\bf 216} (2012), no. 11, 2361--2376. 
MR2927172, Zbl 06138018.

\bibitem{modular}
{\scshape Behrens, Mark}.
\newblock A modular description of the {$K(2)$}-local sphere at the prime 3.
\newblock {\it Topology} {\bf 45} (2006), no. 2, 343--402. 
MR2193339 (2006i:55016), Zbl 1099.55002.

\bibitem{joint}
{\scshape Behrens, Mark; Davis, Daniel~G.}
\newblock The homotopy fixed point spectra of profinite {G}alois extensions.
\newblock {\it Trans. Amer. Math. Soc.} {\bf 362} (2010), no. 9, 4983--5042. MR2645058 (2011e:55016), 
Zbl 1204.55007.

\bibitem{cts}
{\scshape Davis, Daniel~G.}
\newblock Homotopy fixed points for {$L\sb {K(n)}(E\sb n\wedge X)$} using the
  continuous action.
\newblock {\it J. Pure Appl. Algebra} {\bf 206} (2006), no. 3, 322--354. 
MR2235364 (2007b:55008), 
Zbl 1103.55005.

\bibitem{fibrantmodel}
{\scshape Davis, Daniel~G.}
\newblock Explicit fibrant replacement for discrete {$G$}-spectra.
\newblock {\it Homology, Homotopy Appl.} {\bf 10} (2008), no. 3, 137--150. 
MR2475620 (2009k:55018), Zbl 1178.55009.

\bibitem{iterated}
{\scshape Davis, Daniel~G.}
\newblock Iterated homotopy fixed points for the {L}ubin-{T}ate spectrum
  \textrm{(with an appendix by {$\mathrm{D}$}avis, {$\mathrm{D}$}.{$\mathrm{G}$}.; 
  {$\mathrm{W}$}ieland, {$\mathrm{B}$}.)}.
\newblock {\it Topology Appl.} {\bf 156} (2009), no. 17, 2881--2898. 
MR2556043 (2010j:55009), Zbl 1185.55008.

\bibitem{agt}
{\scshape Davis, Daniel~G.}
\newblock Delta-discrete {$G$}-spectra and iterated homotopy fixed points.
\newblock {\it Algebr. Geom. Topol.} {\bf 11} (2011), no. 5, 2775--2814. 
MR2846911 (2012j:55009), Zbl 1230.55006.

\bibitem{jointwithTylertwo}
{\scshape Davis, Daniel~G.; Lawson, Tyler.}
\newblock A descent spectral sequence for arbitrary {$K(n)$}--local spectra
  with explicit {$E_2$}--term.
\newblock {\it Glasg. Math. J.}, 12 pages, doi: 10.1017/S001708951300030X,
  published online by {\it Cambridge University Press}, August 13, 2013.

\bibitem{DH}
{\scshape Devinatz, Ethan~S.; Hopkins, Michael~J.}
\newblock Homotopy fixed point spectra for closed subgroups of the {M}orava
  stabilizer groups.
\newblock {\it Topology} {\bf 43} (2004), no. 1, 1--47. 
MR2030586 (2004i:55012), Zbl 1047.55004.

\bibitem{Pgg/Hop0}
{\scshape Goerss, P.G.; Hopkins, M.J.}
\newblock Moduli spaces of commutative ring spectra.
\newblock {\it Structured ring spectra}, 151--200. London Math.
  Soc. Lecture Note Ser., 315. {\it Cambridge Univ. Press, Cambridge},
  2004. MR2125040 (2006b:55010), Zbl 1086.55006.

\bibitem{GJlocal}
{\scshape Goerss, P.G.; Jardine, J.F.}
\newblock Localization theories for simplicial presheaves.
\newblock {\it Canad. J. Math.} {\bf 50} (1998), no. 5, 1048--1089. 
MR1650938 (99j:55009), Zbl 0914.55004.

\bibitem{hGal}
{\scshape Goerss, Paul~G.}
\newblock Homotopy fixed points for {G}alois groups.
\newblock {\it The \v Cech centennial (Boston, MA, 1993)}, 187--224.
  {\it Amer. Math. Soc., Providence, RI}, 1995. MR1320993 (96a:55008), Zbl 0830.55002.

\bibitem{hirschhorn}
{\scshape Hirschhorn, Philip~S.}
\newblock Model categories and their localizations. 
Mathematical Surveys and Monographs, 99.
\newblock {\it American Mathematical Society, Providence, RI}, 2003. MR1944041 (2003j:18018), 
Zbl 1017.55001. 

\bibitem{jardinejpaa}
{\scshape Jardine, J.F}.
\newblock Simplicial presheaves.
\newblock {\it J. Pure Appl. Algebra} {\bf 47} (1987), no. 1, 35--87. 
MR0906403 (88j:18005), Zbl 0624.18007.

\bibitem{Jardine}
{\scshape Jardine, J.F.}
\newblock Generalized \'etale cohomology theories.
\newblock {\it Birkh\"auser Verlag, Basel}, 1997. MR1437604 (98c:55013), Zbl 0868.19003.

\bibitem{lawsontopology}
{\scshape Lawson, Tyler; Naumann, Niko.}
\newblock Commutativity conditions for truncated {B}rown-{P}eterson spectra of
  height 2.
\newblock {\it J. Topol.} {\bf 5} (2012), no. 1, 137--168. 
MR2897051, Zbl 06016999.

\bibitem{lms}
{\scshape Lewis, L.G., Jr.; May, J.P.; Steinberger, M.; McClure, J.E.}
\newblock Equivariant stable homotopy theory.
\newblock With contributions by McClure, J.E.
\newblock {\it Springer-Verlag, Berlin}, 1986. MR0866482 (88e:55002), Zbl 0611.55001.

\bibitem{luriebook}
{\scshape Lurie, Jacob.}
\newblock Higher topos theory. Annals of Mathematics Studies, 170. 
\newblock {\it Princeton University Press, Princeton, NJ}, 2009. MR2522659 (2010j:18001), 
Zbl 1175.18001.

\bibitem{mitchell}
{\scshape Mitchell, Stephen~A.}
\newblock Hypercohomology spectra and {T}homason's descent theorem.
\newblock {\it Algebraic $K$-theory (Toronto, ON, 1996)}, 221--277.
  {\it Amer. Math. Soc., Providence, RI}, 1997. MR1466977 (99f:19002), Zbl 0888.19003.

\bibitem{Morava}
{\scshape Morava, Jack.}
\newblock Noetherian localisations of categories of cobordism comodules.
\newblock {\it Ann. of Math. (2)} {\bf 121} (1985), no. 1, 1--39. 
MR0782555 (86g:55004), 
Zbl 0572.55005.

\bibitem{QuickSpaces}
{\scshape Quick, Gereon.}
\newblock Continuous group actions on profinite spaces.
\newblock {\it J. Pure Appl. Algebra} {\bf 215} (2011), no. 5, 1024--1039. 
MR2747236 (2012k:55018), 
Zbl 1227.55013.

\bibitem{hfplt}
{\scshape Quick, Gereon.}
\newblock Continuous homotopy fixed points for {L}ubin-{T}ate spectra.
\newblock {\it Homology, Homotopy Appl.} {\bf 15} (2013), no. 1, 191--222. 
MR3079204, Zbl 06169093.

\bibitem{Quick}
{\scshape Quick, Gereon.}
\newblock Profinite {$G$}-spectra.
\newblock {\it Homology, Homotopy Appl.} {\bf 15} (2013), no. 1, 151--189. 
MR3079203, Zbl 06169092.

\bibitem{Rognes}
{\scshape Rognes, John.}
\newblock Galois extensions of structured ring spectra. {S}tably
  dualizable groups. 
\newblock {\it  Mem. Amer. Math. Soc.} {\bf 192} (2008), no. 898, viii+137. 
MR2387923 (2009c:55007), Zbl 1166.55001.

\bibitem{thomason}
{\scshape Thomason, R.W.}
\newblock Algebraic ${K}$-theory and \'etale cohomology.
\newblock {\it Ann. Sci. \'Ecole Norm. Sup. (4)} {\bf 18} (1985), no. 3, 437--552. 
MR0826102 (87k:14016), 
Zbl 0596.14012.

\bibitem{stavros}
{\scshape Tsalidis, Stavros.}
\newblock On the \'etale descent problem for topological cyclic homology and
  algebraic {$K$}-theory.
\newblock {\it $K$-Theory} {\bf 21} (2000), no. 2, 151--199. 
MR1804560 (2001m:19005), Zbl 0980.19002.

\bibitem{Wilson}
{\scshape Wilson, John~S.}
\newblock Profinite groups.
\newblock {\it The Clarendon Press, Oxford University Press, New York}, 1998. 
 MR1691054 (2000j:20048), Zbl 0909.20001.

\end{thebibliography}
\end{document}